\documentclass[12pt]{article}
\usepackage[cp1250]{inputenc}

\usepackage{amsmath}
\usepackage{amsfonts}
\usepackage{amsthm}
\usepackage{amssymb}
\usepackage{color}

\newcommand{\barint}{
         \rule[.036in]{.12in}{.009in}\kern-.16in
          \displaystyle\int  }
\newcommand{\R}{{\mathbb{R}}}

\newcommand{\rn}{{\mathbb{R}^{n}}}

\newcommand{\Om}{{\Omega}}

\newcommand{\al}{{\alpha}}
\newcommand{\be}{{\beta}}

\newcommand{\s}{{\sigma}}

\newcommand{\is}{{\langle\nabla p(x), x\rangle}}

\newtheorem{theo}{\bf Theorem}[section]
\newtheorem{coro}{\bf Corollary}[section]

\newtheorem{rem}{\bf Remark}[section]
\newtheorem{defi}{\bf Definition}[section]

\newtheorem{fact}{\bf Fact}[section]

\def\essinf{\mathop{\textrm{ess\,inf}}}
\def\esssup{\mathop{\textrm{ess\,sup}}}

\title{Variable exponent Hardy--type inequalities\\ in $\rn$}

\newcommand{\comment}[1]{} 

\usepackage{authblk}

\author[$\dagger$]{Sylwia Dudek}
\author[$\diamond$]{Iwona Skrzypczak~\thanks{The author was supported by NCN grant 2011/03/N/ST1/00111.}}

\affil[$\dagger$]{\small
Institute of Mathematics,
Krakow University of Technology, \newline
ul. Warszawska 24, 31--155 Krakow, Poland \newline e--mail: sbarnas@pk.edu.pl}

\affil[$\diamond$]{\small
Faculty of Mathematics, Informatics and Mechanics,
University of Warsaw, \newline
ul. Banacha 2, 02--097 Warsaw, Poland \newline e--mail: iskrzypczak@mimuw.edu.pl}
\date{}

\begin{document}
\maketitle
\thispagestyle{empty}

\begin{abstract}
 In this paper, we develop the results obtained recently by the authors in~\cite{barskrzy1}.  We investigate the weighted $p(x)$--Hardy inequality with the additional term of the form
\[
\int_\Om \ |\xi|^{p(x)}\mu_{1,\be} (dx) \leqslant \int_\Om |\nabla \xi|^{p(x)}\mu_{2,\be} (dx)+\int_\Om \left|\xi{\log \xi} \right|^{p(x)} \mu_{3,\be} (dx),
\]
holding for Lipschitz functions compactly supported in $\Omega\subseteq\rn$. We focus on the $n$--dimensional case giving some examples. Moreover, we compare our inequalities with the existing in the literature.
\end{abstract}

\noindent {\bf Keywords:} \  $p(x)$--Laplacian, Caccioppoli inequality, Hardy inequality, variable exponent Lebesgue space.

\noindent {\bf 2010 Mathematics Subject Classification:} 26D10, 35J60, 35J91.
\newpage

\section{Introduction}\label{intro}

By $p(x)$--harmonic problems we understand those which involve $p(x)$--Laplace operator $\Delta_{p(x)} u=\mathrm{div}(|\nabla u|^{p(x)-2}\nabla u)$. 
Let $\Omega$ be a given open subset of $\mathbb{R}^n$, not necessarily bounded. 
We assume that the function $p$ is such that $p \in W_{loc}^{1,1}(\Omega)$, ${p}^{p(x)},|\nabla p|^{p(x)}\in L^{1}_{loc}(\Omega)$ and satisfies $1<\textrm{ess}\inf_{x \in \Omega} p(x) \leqslant p(x) \leqslant \textrm{ess}\sup_{x \in \Omega} p(x)<\infty$.

In \cite{barskrzy1}  the authors proved the inequality holding for every Lipschitz function $\xi:\Omega \to \mathbb{R}$ with compact support in $\Om$ having the following form
\begin{equation}\label{IntroHar}
\int_\Om \ |\xi|^{p(x)} \mu_{1,\be}(dx)\leqslant\int_\Om \left(|\nabla \xi|^{p(x)}+\left|\xi {\log \xi } \right|^{p(x)}\cdot \frac{\left|\nabla p(x)\right|^{p(x)}}{{p(x)}^{p(x)}}\right) \mu_{2,\be}(dx) ,
\end{equation}
 where the measures $\mu_{1,\be}(dx),\mu_{2,\be}(dx)$ depend on the nonnegative solution to nonlinear problem $-\Delta_{p(x)}u\geqslant \Phi$ with a locally integrable function $\Phi$ (we recall the result as Theorem~\ref{theoplapx} here), and on a certain parameter $\be>0$. 

We investigate futher \eqref{IntroHar} paying special attention to  $n$--dimensional domains $\Omega$ ($n\geqslant 1$). In the general approach we do not require any kind of symmetry of $u$, $p$, or $\Omega$. We present the example of Hardy--type inequality in a general non--radial  case. However, when we assume that $u$ is radially symmetrical, its $p(x)$--Laplacian has simpler form (see Fact~\ref{lemradx}) and the measures are much easier to compute (see Theorem~\ref{theoradx}). We give certain examples. We stress that we do not expect $p$ to be radial. 


We deal with the variable exponent Lebesgue spaces, which recently have received more and more attention both --- from the theoretical and from the applied point of view. We refer to books~\cite{ks,KoRa} for the detailed information on the theoretical approach to the Lebesgue and the Sobolev spaces with variable exponents. Various attempts to prove existence, uniqueness, or regularity theory for problems stated in variable exponent spaces are to found e.g.~in~\cite{barnas, fanzhang}. We refer for the survey \cite{overview}  summarising inter alia results on qualitative properties of solutions to the related PDEs.

 The typical examples of equations stated in variable exponent spaces are models of electrorheological fluids, see e.g.~\cite{raj-ru1,el-rh2}. 
Electrorheological fluids change their mechanical properties dramatically when an external electric field is applied, so variable exponent Lebesgue spaces are natural settings for their modelling. 
Some classical models are also generalised in variable exponent Lebesgue spaces. In \cite{ks} we find investigations on Poisson equation, as well as Stokes problem being of fundamental importance in describing fluid dynamics. 


Hardy--type inequalities are important tools in various fields of analysis.  Let us mention such branches as functional analysis, harmonic analysis, probability theory, and PDEs.  Weighted versions of Hardy--type inequalities are also investigated on their own in the classical way~\cite{kmp,muckenhoupt,plap}, as well as in the various generalised frameworks~\cite{bogdan,bhs,akkpp2012,orliczhardy}.

Recently,  Hardy--type inequalities in variable exponent Lebesgue spaces have become a lively studied topic of analysis~\cite{DiSa,HaHaKo,MCO,MCMO,MaHa,Samko1}. However, they are usually considered in one dimension, and there are only a few $n$--dimensional results. The paper~\cite{Samko1} is devoted to the inequality with the weights depending on distance from a single point, while in~\cite{HaHaKo} the weights depend on distance from a boundary in $\R^n$. 

We point out that in the majority of the above mentioned papers the authors deal with the norm version of Hardy inequalities. We obtain the modular one. We would like to stress that only in the constant exponent case the both types are equivalent. In the variable exponent case it is not direct to transform one of these types to another. To the authors' best knowledge the only result of this kind is given by Fan--Zhao~\cite[Theorem~1.3]{orlicz} where the authors derive a tool giving certain form of the norm version of Hardy inequality from a modular one.

The tool we investigate in this paper is the general $p(x)$--Hardy inequality, which was introduced by  Barna\'s--Skrzypczak~\cite{barskrzy1}. Its proof is based  on  the methods from~\cite{nonex,pohmi_99} developed in~\cite{plap,bcp-plap,orliczhardy}. The special cases of the general $p(x)$--Hardy inequality~\eqref{IntroHar} are classical Hardy and Hardy--Poincar\'{e} inequalities with optimal constants. 

The paper is organised as follows. In Section~\ref{prelim} we introduce tools.  In Section~\ref{ex} we  give examples. Section~\ref{links} is devoted to links with literature. 

\section{Preliminaries}\label{prelim}
\subsubsection*{Notation}

In the sequel we assume that $\Omega\subseteq\mathbb{R}^n$ is an open subset not necessarily bounded. If $f$ is defined on the set $A$ by $f\chi_{A}$ we understand function $f$ extended by $0$ outside $A$.
By $\langle\cdot,\cdot\rangle$ we understand the classical scalar product in~$\R^n$. 


\subsubsection*{General Lebesgue and Sobolev spaces}

In the sequel we suppose that measurable function $p:\Omega \rightarrow (1, \infty)$ is such that
	\begin{equation}\label{P}
  1<p^-:=\essinf_{x \in \Omega} p(x) \leqslant p(x) \leqslant p^+:=\esssup_{x \in \Omega} p(x)<\infty.
	\end{equation}

We recall some properties of the variable exponent
 spaces $L^{p(x)}(\Omega)$ and $W^{1,p(x)}(\Omega)$. By
$E(\Omega)$ we denote the set of all equivalence classes of measurable real functions defined on $\Omega$ being equal almost everywhere. The variable exponent Lebesgue space is defined as
\[
    L^{p(x)}(\Omega)=\{u \in E(\Omega): \int_{\Omega} |u(x)|^{p(x)} dx<\infty \}
\]
  equipped with the norm
$
    \|u\|_{L^{p(x)}(\Omega)}:=\inf \Big\{\lambda>0: \int_{\Omega} \big|\frac{u(x)}{\lambda}\big|^{p(x)}dx \leqslant 1 \Big\}.
$\\
We define the variable exponent Sobolev space $W^{1,p(x)}(\Omega)$ by
\[
    W^{1,p(x)}(\Omega) = \{u \in L^{p(x)}(\Omega): \nabla u \in L^{p(x)}(\Omega; \mathbb{R}^n)\}
\]
  equipped with the norm $\|u\|_{W^{1,p(x)}(\Omega)} = \|u\|_{L^{p(x)}(\Omega)}+\|\nabla u\|_{L^{p(x)}(\Omega)}.$

  Then $(L^{p(x)}(\Omega),\|\cdot\|_{L^{p(x)}(\Omega)})$ and $(W^{1,p(x)}(\Omega), \|\cdot\|_{W^{1,p(x)}(\Omega)})$ are
  separable and reflexive Banach spaces.
  
  For more detailed information we refer to the monographs \cite{ks,fan,orlicz}.

\medskip

 By $\mathcal{P}(\Omega)$	we denote the class of the functions $p$ such that~\eqref{P} is satisfied and $p \in W_{loc}^{1,1}(\Omega)$, ${p}^{p(x)},|\nabla p|^{p(x)}\in L^{1}_{loc}(\Omega)$.

\subsubsection*{Differential inequality}

Our analysis is based on the following differential inequality.

\begin{defi}\label{defnier}
Let $\Om$ be any open subset of $\rn$. We assume that the measurable function $p:\Omega \rightarrow (1, \infty)$ satisfies~\eqref{P}
and $\Phi$ is the locally integrable  function defined in $\Om$ such that  for every nonnegative compactly supported $w\in
W^{1,p(x)}(\Om)$, we have
$
 \int_\Om \Phi w\,dx >-\infty.
$\\
  Let $u\in W^{1,p(x)}_{loc}(\Om)$ and $u\not\equiv 0$. We say that
\[
-\Delta_{p(x)} u\geqslant \Phi,
\]
if for every nonnegative compactly supported $w\in
W^{1,p(x)}(\Om)$, we have
\begin{equation}\label{nikfo}
\langle -\Delta_{p(x)} u,w \rangle := \int_\Om |\nabla u|^{p(x)-2}\langle\nabla
u,\nabla w\rangle\, dx \geqslant \int_\Om \Phi w\, dx.
\end{equation}
\end{defi}
\begin{rem}\label{mal1}\rm
Note that $p(x)$--Laplacian is a continuous, bounded, and strictly monotone operator defined for every compactly supported function $w \in W^{1,p(x)}(\Om)$ (see e.g. \cite[Theorem 3.1]{fanzhang} for the definitions and the proofs).
In particular, it is well--defined in the distributional sense.
\end{rem}

When we consider radial solutions to $-\Delta_{p(x)}u\geqslant \Phi$, we need the following useful fact, whose proof is given in Appendix.

\begin{fact}[$p(x)$--Laplacian of a radial function]\label{lemradx} Let $\Omega\subseteq \mathbb{R}^n$ be an open set and $u(x)=v(|x|)\in W^{1,p(x)}_{loc}(\Omega)\cap W^{2,1}_{loc}(\Omega)$, then
\begin{eqnarray*}
&\Delta_{p(x)} u(x)=&\\
&=|v'(|x|)|^{p(x)-2}\Big[ \langle\nabla p(x), x\rangle v'(|x|) \frac{\log |v'(|x|)| }{|x| }+{v''(|x|)}(p(x)-1)+\frac{(n-1)v'(|x|)}{|x|}\Big].&\end{eqnarray*}
\end{fact}

\subsubsection*{Crucial conditions} 

We suppose that the measurable function $p:\Omega \rightarrow (1,\infty)$ satisfies~\eqref{P}, nonnegative $u
\in W^{1,p(x)}_{loc}(\Om)$ and $\Phi\in L^{1}_{loc}(\Om)$ satisfy PDI $-\Delta_{p(x)} u\geqslant \Phi$, in the sense of Definition~\ref{defnier}. 
 We assume that there exist a continuous function $\s(x):\overline{\Omega}\to\R$ and a parameter $\beta>0$, such that the following conditions are satisfied
\begin{eqnarray}
&\label{sx}{\Phi\cdot u}+{\s(x) |\nabla u|^{p(x)}}\geqslant 0 \quad{\rm a.e.\  in}\quad \Om,&\\
&\be>\sup\limits_{x \in \overline{\Omega}} \s(x).&\label{sbeta}
\end{eqnarray}


\subsubsection*{General variable exponent inequality}\label{mainpx}

We investigate further the following  main result of Barna\'s--Skrzypczak \cite{barskrzy1}.

\begin{theo} [\cite{barskrzy1}, Theorem 4.1]
\label{theoplapx} 
Let $\Omega\subseteq\mathbb{R}^n$ be an open subset not necessarily bounded and $p\in{\cal{P}}(\Omega)$.
Let nonnegative $u
\in W^{1,p(x)}_{loc}(\Om)$ and $\Phi\in L^{1}_{loc}(\Om)$ satisfy PDI $-\Delta_{p(x)} u\geqslant \Phi$, in the sense of Definition~\ref{defnier}.
Assume further that functions $u$, $\Phi$, $p(x)$, $\s(x)$ and a parameter $\beta>0$ satisfy Crucial conditions~\eqref{sx}
and~\eqref{sbeta}. 

Then for every Lipschitz function $\xi$ with compact support in $\Om$ we have
\begin{equation}\label{hardypx1}
\int_\Om \ |\xi|^{p(x)} \mu_{1,\beta}(dx)\leqslant \int_\Om |\nabla \xi|^{p(x)}\mu_{2,\beta}(dx)+\int_\Om \left|\xi
 {\log \xi } \right|^{p(x)}\cdot \frac{\left|\nabla p(x)\right|^{p(x)}}{{p(x)}^{p(x)}} \mu_{2,\beta}(dx),
\end{equation}
where
\begin{eqnarray}\label{mu1px}
&\mu_{1,\beta}(dx)&=\big(\Phi\cdot u+\sigma(x) |\nabla u|^{p(x)}\big)\cdot u^{-\be-1}\chi_{\{u>0\}}\ dx,
\\\label{mu2px}
&\mu_{2,\beta}(dx)&=   { \Big(\frac{p(x)-1}{\be-\s(x)}\Big)^{p(x)-1} } 2^{(p(x)-1)\chi_{\left\{|\nabla p|\neq 0\right\}}}u^{p(x)-\be-1}\chi_{\{|\nabla u|\neq 0\}}\ dx.\end{eqnarray}
\end{theo}

\begin{rem}\rm

We note that \eqref{hardypx1} is of Hardy type with respect to $\xi$ and of Caccioppoli type with respect to $u$. Indeed, in the terms of $\xi$ --- we find $|\xi|^{p(x)}$ on the left--hand side of \eqref{hardypx1} and $|\nabla \xi|^{p(x)}$ on the right--hand side. Caccioppoli inequality should involve $|\nabla u|^{p(x)}$  on the left--hand side and $u^{p(x)}$ on the right--hand side. To be precise and avoid dependence on the gradient on the right--hand side of~\eqref{IntroHar}, one should estimate the characteristic function ($\chi_{\{|\nabla u|\neq 0\}}\leqslant 1$).

The paper~\cite{barskrzyCac} is devoted to analysis of~\eqref{hardypx1} as the Caccioppoli inequality,
 and Liouville--type result for solutions to  $-\Delta_{p(x)} u\geqslant \Phi$ are given therein. 
\end{rem}

\begin{rem}\rm \label{remCacHar}
When we consider $1<p(x)\equiv p<\infty$, we retrieve the main result of  \cite{plap}, implying the classical Hardy inequality with optimal constant (see \cite{plap} for the details and some other examples). Moreover, it gives optimal constants for Hardy--Poincar\'{e} inequalities with weights of a type $\left(1+|x|^\frac{p}{p-1}\right)^\al$ involving the sufficiently big parameter $\al>0$ (see \cite{bcp-plap} for the details). 
\end{rem}


\subsubsection*{Quasi--radial inequality}

When we assume that the nonnegative function $u \in W^{1,p(x)}_{loc}(\Om)$, which is supposed to satisfy $-\Delta_{p(x)}u\geqslant\Phi$, is a radial function, we may simplify the statement of Theorem~\ref{theoplapx}. We point out that we do not expect $p(x)$ to be radial.  For this reason  we call this case quasi--radial.
We remark that we start Subsection~\ref{radial} with the example of radial $u$ satisfying the PDI $-\Delta_{p(x)}u\geqslant\Phi$ with non--radial $p$. The proof of the following theorem is given in Appendix.

\begin{theo}[Inequality with quasi--radial measures]
\label{theoradx}
Assume that $\Omega \subseteq \mathbb{R}^n$ is an open subset, $p \in \mathcal{P}(\Omega)$, and $u(x)=v(|x|)\in W^{1,p(x)}_{loc}(\Omega)\cap W^{2,1}_{loc}(\Omega)$ is a nonnegative function. 
Assume further that a continuous  function $\s(x)$ and a parameter $\beta>0$ are such that  $\sup_{x \in \overline{\Omega}}\s(x)<\beta$ and the following condition is satisfied
\[
K(x):=\s(x)- \frac{v(|x|)}{v'(|x|)}\Big[ \is \frac{\log|v'(|x|)| }{|x| }+\frac{v''(|x|)}{v'(|x|)}(p(x)-1)+\frac{n-1}{|x|}\Big] \geqslant 0.
\]
Then, for every Lipschitz function $\xi$ with compact support in $\Omega$, we have
\begin{equation*}
\int_\Om \ |\xi|^{p(x)} \mu_{1,\be}(dx)\leqslant \int_\Om |\nabla \xi|^{p(x)}\mu_{2,\be}(dx)+\int_\Om \left|\xi {\log \xi } \right|^{p(x)}\cdot \frac{\left|\nabla p(x)\right|^{p(x)}}{{p(x)}^{p(x)}} \mu_{2,\be}(dx),
\end{equation*}
where
\begin{eqnarray*}\label{mu1radx}
&\mu_{1,\be}(dx)&= |v'(|x|)|^{p(x)}(v(|x|))^{-\beta-1}\chi_{\{v>0\}}
K(x) dx,\\
\label{mu2radx}&\mu_{2,\be}(dx)&=   { \Big(\frac{p(x)-1}{\beta-\s(x)}\Big)^{p(x)-1} } 2^{(p(x)-1)\chi_{\left\{|\nabla p|\neq 0\right\}}}(v(|x|))^{p(x)-\beta-1}\chi_{\{|v'|\neq 0\}} dx.\end{eqnarray*}
\end{theo}
\begin{rem}\rm The above inequality should be called Caccioppoli inequality for radial solution $u(x)=v(|x|)$ to  $-\Delta_{p(x)}u\geqslant\Phi$ and weighted Hardy inequality for Lipschitz and compactly supported functions $\xi$ (with not necessarily radial weights). For discussion see Remark~\ref{remCacHar}. \end{rem}


\section{Examples of variable exponent Hardy inequality}\label{ex}

This section is devoted to original examples of \eqref{hardypx1}.
The first part deals  with  several applications of Theorem~\ref{theoradx}. Let us recall that we say it is quasi--radial case, because we expect here radial $u$ (solving $-\Delta_{p(x)}u\geqslant \Phi$), but not necessarily $p$.  The second part of this section contains a non--radial example. We give it here to stress, that our method is general and we do not need simplifications such us one--dimensionality or radiality.

\subsection{Quasi--radial case}\label{radial}

Before we present a sample inequalities resulting  
from Theorem~\ref{theoradx} we point out that quasi--radial case is  
not empty. Indeed, it may happen that PDI  
$-\Delta_{p(x)}u\geqslant\Phi$ with non--radial $p \in \mathcal{P}(\Omega)$ is satisfied by nonnegative
radial function $u \in W^{1,p(x)}_{loc}(\Omega)$.
Let us give a one--dimensional example.
\begin{rem}
  On the bounded interval $I \subseteq (-M,M)\subseteq\R$, with some $M>0$ function  
$u(x)=-e(|x|+M)$ is a radial solution to the PDE $-\Delta_{p(x)}u=\Phi$ with
\[p(x)=\left\{\begin{array}{ll}2+\frac{1}{1-x}&x<0,\\
5-\frac{4}{x+2}&x\geqslant 0.\end{array}\right.\quad\quad  
\Phi=\left\{\begin{array}{ll}-\frac{1}{(1-x)^2}e^{p(x)-1}&x<0,\\
-\frac{4}{(x+2)^2}e^{p(x)-1}&x\geqslant 0.\end{array}\right.\] To  
satisfy the rest of restrictions from Theorem~\ref{theoradx} we take
\[\sigma(x)=\left\{\begin{array}{ll}2\frac{x-M}{(1-x)^2}&x<0,\\
-8\frac{x+M}{(x+2)^2}&x\geqslant 0.\end{array}\right.\quad and \quad \beta>0.\]
 We note that $p(x)\in{\cal{P}}(I)$  and $\sigma(x)$ is continuous. 
\end{rem}
Let us present inequalities with power--type weights and  with  
exponential--type weights as applications of Theorem~\ref{theoradx}.

\subsubsection*{Inequalities with power--type weights}

When we apply $u(x)=|x|$ in Theorem~\ref{theoradx}, we obtain the following result.
\begin{coro}
\label{coro|x|} Let $\Omega\subseteq\R^n\setminus\{0\}$.
Suppose that $p \in \mathcal{P}(\Omega)$, a continuous  function $\s(x)$ and a parameter $\beta>0$ are such that  $\sup_{x \in \overline{\Omega}}\s(x)<\beta$ and
 $\s(x)\geqslant n-1$.

Then, for every Lipschitz function $\xi$ with compact support in $\Omega$, we have
\begin{equation}\label{hardypx|x|}
\int\limits_{\Omega} \ |\xi|^{p(x)} \mu_{1,\be}(dx)\leqslant \int\limits_{\Omega} |\nabla \xi|^{p(x)}\mu_{2,\be}(dx)+\int\limits_{\Omega} \left|\xi {\log \xi } \right|^{p(x)}\cdot \frac{\left|\nabla p(x)\right|^{p(x)}}{{p(x)}^{p(x)}} \mu_{2,\be}(dx),
\end{equation}
where
\begin{eqnarray*}
\mu_{1,\be}(dx)&=&|x|^{-\be-1}\,\big(\sigma(x)+1-n\big)\; dx,\\
 \mu_{2,\be}(dx)&= & |x|^{p(x)-\be-1}  { \Big(2 \cdot \frac{p(x)-1}{\be-\s(x)}\Big)^{p(x)-1} }  \ dx.
\end{eqnarray*}
\end{coro}

\begin{rem}\rm
When we choose $\s (x)=\be-2(p(x)-1)$ in Corollary~\ref{coro|x|} and require that  $p \in \mathcal{P}(\Omega)$ satisfies $
p^+< \frac{\be-n+3}{2}$ then, for $\overline{s}=\frac{\be-n+3}{2}-p^+$, we obtain \eqref{hardypx|x|}
with
\[
\mu_{1,\be}(dx)=\overline{s} |x|^{-\be-1} dx,\]
\[\mu_{2,\be}(dx)= |x|^{p(x)-\be-1} \; dx.
\]
\end{rem}

\bigskip

\noindent Before we present more complex example, let us define
\[
K_\alpha(x):=\s (x)-\frac{1}{\alpha}\left[(\alpha-1) \big(\is  \log |x|  +p(x)\big)+n -\alpha\right].
\]

\begin{coro}
\label{theoplapxpower}Let $\Omega\subseteq\R^n\setminus\{0\}$.
Suppose that $p \in \mathcal{P}(\Omega)$ is such that the function $\is\log |x|$ is bounded from below (from above) by $C_L\in\R$
and $\alpha\geqslant 1$ ($0<\alpha<1$) is a given constant. 
Assume further that $\s(x)$ and $\beta>0$ are such that  $\sup_{x \in \overline{\Omega}}\s(x)<\beta$ and the following condition is satisfied
 \[\alpha\sigma(x)-n+\alpha-p(x)(\alpha-1)\geqslant (\alpha-1)C_L.\]
Then, for every Lipschitz function $\xi$ with compact support in $\Omega$, we have
\[
\int\limits_{\Omega} \ |\xi|^{p(x)} \mu_{1,\be}(dx)\leqslant \int\limits_{\Omega} |\nabla \xi|^{p(x)}\mu_{2,\be}(dx)+\int\limits_{\Omega} \left|\xi {\log \xi } \right|^{p(x)}\cdot \frac{\left|\nabla p(x)\right|^{p(x)}}{{p(x)}^{p(x)}} \mu_{2,\be}(dx),
\]
where
\begin{eqnarray*}
\label{mu1pxpower}
\mu_{1,\be}(dx)&=& |x| ^{\alpha(p(x)-\be-1)-p(x)}  \alpha \cdot K_{\alpha}(x)\; dx\\
 \mu_{2,\be}(dx)&= &  |x| ^{\alpha(p(x)-\be-1)}  { \Big(\frac{2}{\alpha} \cdot \frac{p(x)-1}{\be-\s(x)}\Big)^{p(x)-1} }  \ dx.
\end{eqnarray*}
\end{coro}

\begin{proof}
We apply Theorem~\ref{theoradx} to $u(x)=\frac{1}{\alpha} |x| ^\alpha$. We note that
\[v(|x|)=\frac{1}{\alpha} |x| ^\alpha, \quad v'(|x|)= |x| ^{\alpha-1}, \quad v''(|x|)= (\alpha-1) |x| ^{\alpha-2}.\]
Direct computations gives desired inequality 
(meanwhile we divide the both sides of the inequality by $\alpha^{\be}$).
\end{proof}

\begin{rem}
We note that when we choose $\sigma(x)=\beta-\frac{2(p(x)-1)}{\alpha}$ in Theorem~\ref{theoplapxpower}, in the case $\alpha > 1$, the only requirement on the exponent  $p \in \mathcal{P}(\Omega)$ is that $
p^+\leqslant \frac{\alpha \be+2-n+\alpha-(\alpha-1)C_L}{\alpha+1}$. Then, for $\overline{s}= \frac{\alpha \be+2-n+\alpha-(\alpha-1)C_L}{\alpha+1}-p^+$, we have Hardy inequality
with
\[\mu_{1,\be}(dx)=\overline{s}\alpha|x|^{\alpha(p(x)-\be-1)-p(x)} \ dx\]\[
\mu_{2,\be}(dx)= |x|^{\alpha(p(x)-\be-1)}  dx.\]
\end{rem}

\subsubsection*{Inequalities with exponential weights}
The other special case of the above quasi--radial inequalities are inequalities with exponential weights. We obtain them directly from Theorem~\ref{theoradx} when we take $u(x)=e^{|x|}$.

 Let us define
\[
K^e(x):= \; \label{sexp}\s(x)-\is -p(x)+1+\frac{1-n}{|x|}.
\]

\begin{coro}
\label{theoplapxpowerexp1}
Let $\Omega \subseteq \mathbb{R}^n$ and  $p \in \mathcal{P}(\Omega)$ is such that for some $C_e>0$ we have that $\is > C_e$. Suppose that a continuous  function $\s(x)$ and a parameter $\beta>0$ are such that  $\sup_{x \in \overline{\Omega}}\s(x)<\beta$  and the following condition is satisfied
\begin{equation*}
|x|\s(x)\geqslant |x| (C_e+p(x)-1)+n-1
\end{equation*}
Then, for every Lipschitz function $\xi$ with compact support in $\Omega$, we have
\[
\int\limits_{\Omega} \ |\xi|^{p(x)} \mu_{1,\be}(dx)\leqslant \int\limits_{\Omega} |\nabla \xi|^{p(x)}\mu_{2,\be}(dx)+\int\limits_{\Omega} \left|\xi {\log \xi } \right|^{p(x)}\cdot \frac{\left|\nabla p(x)\right|^{p(x)}}{{p(x)}^{p(x)}} \mu_{2,\be}(dx),
\]
where
\begin{eqnarray*}
\mu_{1,\be}(dx)&=&e^{|x|(p(x)-\be-1)} \cdot K^e(x) \; dx\\
 \mu_{2,\be}(dx)&= &{ e^{|x|(p(x)-\be-1)}\cdot\Big(2\frac{p(x)-1}{\be-\s(x)}\Big)^{p(x)-1} }\, dx.
\end{eqnarray*}
\end{coro}

\begin{rem}
For $\s(x)=\be - 2(p(x)-1)$ in Corollary~\ref{theoplapxpowerexp1}, the only requirement  on the exponent $p \in \mathcal{P}(\Omega)$ is that $p^+ \leqslant \frac{\be+C_e}{3}+1$. Thus, for $\overline{k}=\frac{\be+C_e}{3}+1 -p^+$, we have our inequality 
with measures
\begin{eqnarray*}
\mu_{1,\be}(dx)&=&\overline{k} e^{|x|(p(x)-\be-1)} \; dx\\
 \mu_{2,\be}(dx)&= & e^{|x|(p(x)-\be-1)}\, d x.
\end{eqnarray*}
\end{rem}


\subsection{Non--radial example}\label{generalex}

To emphasize that we do not have to be restricted to the case of radial function $u$, in this section we present application of Theorem~\ref{theoplapx} in the general settings. The computations are given in Appendix.

\begin{coro}\label{generalcase}
Let $\Omega=(\R_+)^n$, $J(x)=\sum\limits_{j=1}^n j x_j$,  $S=\frac{n}{6} (2n+1)(n+1)$, $\beta>0$ is an arbitrary number, $T(x)=-\frac{S\beta}{J(x)+\frac{\log S}{2}}$. Assume $p\in{\cal P}(\Omega)$ and  $\sum\limits_{j=1}^n j \frac{\partial p}{\partial x_j}<T(x)$ a.e. in $\Omega$.\\
Then, for every Lipschitz function $\xi$ with compact support in $\Omega$, we have
\[
\int\limits_{\Omega} \ |\xi|^{p(x)} \mu_{1,\be}(dx)\leqslant \int\limits_{\Omega} |\nabla \xi|^{p(x)}\mu_{2,\be}(dx)+\int\limits_{\Omega} \left|\xi {\log \xi } \right|^{p(x)}\cdot \frac{\left|\nabla p(x)\right|^{p(x)}}{{p(x)}^{p(x)}} \mu_{2,\be}(dx),
\]
where
\begin{eqnarray*}
\mu_{1,\be}(dx)&=&e^{(p(x)-\be-1)J(x)} \cdot S^{\frac{p(x)}{2}} \left(\beta-\sum\limits_{j=1}^n j \frac{\partial p}{\partial x_j} \frac{J(x)+\frac{\log S}{2}}{S}\right) \; dx\\
 \mu_{2,\be}(dx)&= &e^{(p(x)-\be-1)J(x)} \cdot  2^{p(x)-1} \, dx.
\end{eqnarray*}
\end{coro}

\begin{rem}In Corollary~\ref{generalcase} we may consider various  functions $p(x)$, e.g.~$p(x)=1+e^{-J(x)}$ or~$p(x)=1+(x+1)^{-J(x)}$.

If we restrict ourselves to bounded domain, it suffices to assume that $\sum\limits_{j=1}^n j \frac{\partial p}{\partial x_j}<T$ with a contant $T$. For instance for $\Omega=[0,M]^n$, we allow $T=\frac{-2S\be}{Mn+Mn^2+\log S}$.

\end{rem}


\section{Links with the existing results}\label{links}

In this section we present how our result is related to the several other inequalities holding over $\Om\subseteq\R^n$. The detailed analysis of the one--dimensional case is given in Barna\'s--Skrzypczak  \cite{barskrzy1}.

We need to introduce the class of locally log--H\"older continuous functions. By $\mathcal{P}^{\log} (\Omega)$ we understand the family of measurable functions $p: \Omega \to (1, \infty)$ satisfying~\eqref{P} and the following condition
\[
|p(x)-p(y)| \leqslant  \frac{c}{\log(e+\frac{1}{|x-y|})} \qquad \textrm{for all } \; x,y \in \Omega.
\] 

\subsubsection*{Results of Harjulehto--H\"{a}st\"{o}--Koskenoja \cite{HaHaKo}}
 
In \cite{HaHaKo} the norm version of Hardy inequality involving weights depen\-dent on distance term from boundary  is shown provided maximal operator is bounded. The main result of this paper reads as follows.

\begin{theo}[{\cite[Theorem 3.3]{HaHaKo}}] Let  $\Omega$  be an open and bounded subset of $\R^n$.  Suppose that $p \in \mathcal{P}^{\log}(\Omega)$. Assume that, if $B(x,r)$ is the open ball with center $x$ and radius $r$ and $\Omega^c$ is the complement of $\Omega$, then there exists a constant $b>0$ satisfying $|B(z,r)\cap \Omega^c| \geqslant  b |B(z,r)|$ for every $z \in \partial \Omega$ and $r>0$. 

 Then, for all $\xi \in W^{1, p(x)}_0 (\Omega)$ there exist constants $C,a_0>0$ (depending only on $p$, $n$ and $b$) such that for all $0\leqslant a <a_0$, we have
 \begin{equation}\label{hhk}
\| \xi(x)\delta^{a-1}(x)\|_{L^{p(x)}(\Om)} \leqslant C \| \nabla \xi(x) \delta^a (x)\|_{L^{p(x)}(\Om)},
\end{equation}
where $\delta(x)={\rm{dist}}( x, \partial \Omega)$.
\end{theo}

The authors point out that if $p^->n$, then Hardy inequality \eqref{hhk}  holds on every
bounded open set~$\Omega$ in~$\R^n$. Additionally, the authors of \cite{HaHaKo} indicate that, if $\omega$ is a modulus of continuity such that $\omega(x) \log \frac{1}{x} \rightarrow \infty$ as $x \rightarrow 0$, then there exists a variable exponent, such that $|p(x)-p(y)| \leqslant \omega (|x-y|)$,
for which Hardy inequality does not hold with $\xi\in W_{0}^{1,p(x)}(\Omega)$.

The main difference between our result and the one of~\cite{HaHaKo} is that \eqref{hhk} is a norm version of Hardy inequality, while we obtain the modular one with the additional term.  We recall that  in $L^{p(x)}(\Omega)$ we deal with  the norm
\[
    \|u\|_{L^{p(x)}(\Omega)}:=\inf \Big\{\lambda>0: \int_{\Omega} \Big|\frac{u(x)}{\lambda}\Big|^{p(x)}dx \leqslant 1 \Big\},
\]
which in the case of nonconstant exponent is not directly comparable with $\int_{\Omega} | u(x)|^{p(x)}dx$ (see e.g. Fan--Zhao \cite[Theorem 1.3]{orlicz}).

Moreover, \eqref{hhk} involves the measures $ {\delta^{a-1}}$ and $\delta^a$ under the norm sign, while in our Theorem~\ref{theoplapx} the internal measures are trivial. On the other hand, in~\eqref{hhk} the external measure on the both sides of inequality is the Lebesgue's measure, while in our case the external measures on the both sides of the inequality are different and nontrivial.

\subsubsection*{Results of Chen \cite{Chen}}

In \cite{Chen}  the norm version of Hardy type inequality is established in weighted variable exponent Sobolev space. Moreover, the results are applied in the proof of the existence of a nontrival weak solutions of a $p(x)$--harmonic equations. The main result of this paper reads as follows. 

\begin{theo}[{\cite[Theorem 3.3]{Chen}}]
Let $p \in \mathcal{P}(\mathbb{R}^n)$.
Suppose that $w(x)=w(|x|),$ $v_0(x)=v_0(|x|),$ $v_1(x)=v_1(|x|)$ are radial weight functions positive, measurable and finite a.e. in~$\rn$. Let us define
\begin{align*}
&L_{w(x)}^{p(x)}(\rn)=\{u \in E(\mathbb{R}^n): \ u w^{\frac{1}{p(x)}} \in L^{p(x)}(\mathbb{R}^n)\},&\\
&W_{v_0(x),v_1(x)}^{1,p(x)} (\rn)=\{u \in E(\mathbb{R}^n): \ u v_0^{\frac{1}{p(x)}} \in L^{p(x)}(\mathbb{R}^n),\ |\nabla u| v_1^{\frac{1}{p(x)}} \in L^{p(x)}(\mathbb{R}^n)\}.&
\end{align*}
Assume further that $v_0(|x|)|x|^{n-1} \in L^1(\mathbb{R}_+)$, $P_1=\{ p^+,p^-, \frac{(p^+)^2}{p^-}\}$ and $P_2=\{p^+(p^+-1), \frac{(p^+)^2(p^--1)}{p^-}, p^-(p^+-1), p^+(p^--1)\}, $  and the following condition 
\begin{equation}\label{condition:dodatkowy}
\sup_{\substack{r_i \in P_i \\ 0<t<\infty}} \Big( \int_0^t (v_1(s) s^{n-1})^{\frac{1}{1-p(sy)}} \ ds \Big)^{\frac{r_2}{(p^+)^2}} \Big(\int_t^\infty v_0 (s) s^{n-1} \ ds \Big)^{\frac{r_1}{(p^+)^2}} <\infty
\end{equation}
 holds for  all $y$ belonging to the unit sphere in $\rn$, 
then there exists a constant $C>0$ such that for every $\xi \in W_{v_0(x),v_1(x)}^{1,p(x)} (\rn)$ whose trace is zero, we have the following Hardy--type inequality 
\begin{equation}\label{hardy2}
\|\xi\|_{L_{v_0(x)}^{p(x)}(\rn)} \le C\|\nabla \xi\|_{{L_{v_1(x)}^{p(x)}(\rn)} }.
\end{equation}
\end{theo}

The major difference between our result and the one of Chen~\cite{Chen} is that~\eqref{hardy2} is a norm version of Hardy inequality, while our is a modular one. In non--constant exponent case, there is no equivalence between this two forms (see the comment on the result of Harjulehto--H\"{a}st\"{o}--Koskenoja~\cite{HaHaKo}).

The assumption~\eqref{condition:dodatkowy} should be called Muckenhoupt--type condition, which is classical when we deal with a constant exponent (see the seminal paper~\cite{muckenhoupt}). This kind of approach is not constructive and therefore it is hardly comparable with our approach.

\subsubsection*{ Results of Rafeiro--Samko \cite{RaSa}}

In \cite{RaSa} the derived Hardy inequality is connected with Riesz potential. It is stated on a bounded domain $\Omega \subseteq \mathbb{R}^n$, which complement $\mathbb{R}^n\backslash \overline{\Omega}$ has the cone property and satisfies some condition involving a parameter $\alpha$ and ${\Omega}$ (for more details see \cite{RaSa}). The main result reads
\begin{equation*}
\Big\|\delta (x) ^{-\alpha} \int_{\Omega} \frac{\xi(y)}{|x-y|^{n-\alpha}}dy\Big\|_{L^{p(x)}(\Om)} \leqslant C\|\xi\|_{L^{p(x)}(\Om)},\quad 0<\alpha<\min \left( 1, \frac{n}{p^+}\right),
\end{equation*}
where the variable exponent $p(x)$ satisfies the log--H\"older condition and $\delta (x)={\rm dist} (x, \partial \Omega)$.
Similar Hardy inequality is considered in~\cite{MaHa}.

\section*{Appendix}

\begin{proof}[Proof of Fact~\ref{lemradx}] In order to compute $\Delta_{p(x)} u(x)=\mathrm{div}(|\nabla u(x)|^{p(x)-2} \nabla u(x))$, we note that
{\small \begin{align*}
\nabla u(x)&=\Big[v'(|x|)\frac{\partial |x|}{\partial x_1}, \ldots,(v'(|x|)\frac{\partial |x|}{\partial x_n} \Big]=\frac{v'(|x|)}{|x|}x,\\
|\nabla u(x)|^{p(x)-2} \nabla u(x) &= |v'(|x|)|^{p(x)-2}\frac{v'(|x|)}{|x|}x,\\
\frac{\partial }{\partial x_j}\Big[|\nabla u(x)|^{p(x)-2} \nabla u(x)\Big]&=\frac{\partial }{\partial x_j}\Big[ |v'(|x|)|^{p(x)-2}\Big] \frac{v'(|x|)}{|x|} {x_j}+|v'(|x|)|^{p(x)-2} \frac{\partial }{\partial x_j}\Big[v'(|x|)\cdot \frac{x_j}{|x|}\Big].\end{align*} }
Therefore,
{\small \begin{align*}\frac{\partial }{\partial x_j}\Big[|v'(|x|)|^{p(x)-2}\Big]&=|v'(|x|)|^{p(x)-2}\Big[ \frac{\partial }{\partial x_j}[p (x)-2] \cdot \log |v'(|x|)|+(p(x)-2)\frac{v''(|x|)}{v'(|x|)}\frac{x_j}{|x|}\Big],\\
\frac{\partial }{\partial x_j}\Big[v'(|x|)\cdot \frac{x_j}{|x|}\Big]&=\frac{\partial }{\partial x_j}[v'(|x|)] \frac{x_j}{|x|}+v'(|x|) \frac{\partial }{\partial x_j}\Big[\frac{x_j}{|x|}\Big]=v''(|x|)\frac{x_j^2}{|x|^2}+v'(|x|) \Big(\frac{1}{|x|}-\frac{x_j^2}{|x|^3} \Big).\end{align*} }
Thus
{\small\begin{eqnarray*}
&\Delta_{p(x)} u(x)=\sum\limits_{j=1}^n \frac{\partial }{\partial x_j} \Big[|\nabla u(x)|^{p(x)-2} \nabla u(x)\Big]=&\\
&|v'(|x|)|^{p(x)-2}\sum\limits_{j=1}^n \Big(  \frac{\frac{\partial }{\partial x_j}p(x) \log |v'(|x|)|v'(|x|)}{|x|}x_j+(p(x)-1)\frac{v''(|x|)}{|x|^2}x^2_j + \frac{v'(|x|)}{|x|}\left[1-\frac{x^2_j}{|x|^2}\right]\Big)=&\\
&|v'(|x|)|^{p(x)-2}v'(|x|)\Big[\is  \frac{\log |v'(|x|)| }{|x| }+\frac{v''(|x|)}{v'(|x|)}(p(x)-1)+\frac{n-1}{|x|}\Big].&
\end{eqnarray*}}\end{proof}

\begin{proof}[Proof of Theorem~\ref{theoradx}]
The result follows from Theorem~\ref{theoplapx} when instead of the PDI  we take PDE $-\Delta_{p(x)} u=\Phi$. We  apply Fact~\ref{lemradx} and  realize that
\[
-\Delta_{p(x)} u \cdot u+\sigma(x)|\nabla u|^{p(x)}=\]
{\small \[=-|v'(|x|)|^{p(x)-2}v'(|x|)v(|x|)\Big[ \is \frac{\log|v'(|x|)| }{|x| }+\frac{v''(|x|)}{v'(|x|)}(p(x)-1)+\frac{n-1}{|x|}\Big]+\]
\[+\s (x)|v'(|x|)|^{p(x)}=\]}
{\small \[=|v'(|x|)|^{p(x)}\left\{ \s(x)- \frac{v(|x|)}{v'(|x|)}\Big[\is \frac{\log |v'(|x|)| }{|x| }+\frac{v''(|x|)}{v'(|x|)}(p(x)-1)+\frac{n-1}{|x|}\Big]\right\}=\]}
{\small \[ =|v'(|x|)|^{p(x)} \cdot K(x) \geqslant 0.\]}\end{proof}

\begin{proof} [Proof of Corollary~\ref{generalcase}]
 We apply Theorem \ref{theoplapx} with the function $u: \mathbb{R}^n \rightarrow \mathbb{R}$ such that $u(x_1, \ldots, x_n)=e^{ J(x)}$. Then
\begin{eqnarray*}
\nabla_j u&=&j e^{ J(x)},\\
|\nabla u|^{p(x)-2} &=&S^{\frac{p(x)-2}{2}} e^{(p(x)-2) J(x)}, \\
|\nabla u|^{p(x)-2} \nabla_j u&=&j S^{\frac{p(x)-2}{2}}  e^{(p(x)-1) J(x)}, \\
\frac{\partial}{\partial x_j}\Big( |\nabla u|^{p(x)-2} \nabla_{j} u\Big) &=&j S^{\frac{p(x)}{2}}  e^{(p(x)-1) J(x)} \Big[ \frac{J(x)+\frac{\log S}{2}}{S} \frac{\partial p}{\partial x_j} +\frac{j}{S}(p(x)-1)\Big],\\
\sum\limits_{j=1}^n \frac{\partial}{\partial x_j}\Big( |\nabla_{j} u|^{p(x)-2} \nabla_{j} u\Big) &=&S^{\frac{p(x)}{2}}  e^{(p(x)-1) J(x)} \cdot\\
&&\cdot \left[ \frac{J(x)+\frac{\log S}{2}}{S} \sum\limits_{j=1}^n j \frac{\partial p}{\partial x_j}+p(x)-1\right].
\end{eqnarray*}

Thus, when we take $\sigma(x)=p(x)+\beta-1$ , we have
\begin{eqnarray*}
&-\Delta_{p(x)} u \cdot u+\s (x) |\nabla u|^{p(x)}&=\\
=& S^{\frac{p(x)}{2}}e^{p(x) J(x)}
\Big[ \sigma(x)+1- p(x)-\frac{J(x)+\frac{\log S}{2}}{S} \sum\limits_{j=1}^n j \frac{\partial p}{\partial x_j}\Big]&=\\
=& S^{\frac{p(x)}{2}}e^{p(x) J(x)}
\Big[\beta-\frac{J(x)+\frac{\log S}{2}}{S} \sum\limits_{j=1}^n j \frac{\partial p}{\partial x_j}\Big],&
\end{eqnarray*}
which is assumed to be nonnegative.\end{proof}

\end{document}